\documentclass{amsart}
\usepackage{graphicx} % Required for inserting images
\usepackage{amssymb, latexsym, mathrsfs, color, tikz, multirow,mathtools,xcolor,enumerate,caption, subcaption, graphicx,xcolor,comment,multicol,amsmath,amsthm}
 \usepackage{caption}
\usepackage{subcaption}
\usepackage{graphics}
\graphicspath{ {./Diagrams/} }
\usepackage{pythonhighlight}

\usepackage{xspace}
\input{xy}
\xyoption{all}

\usepackage[colorlinks=true, pdfstartview=FitV,
 linkcolor=blue,citecolor=blue,urlcolor=blue]{hyperref}
\usetikzlibrary{shapes,snakes}

\usepackage{ytableau} %Young Tableau

\numberwithin{equation}{section}
\newtheorem{theorem}[equation]{Theorem}

\newtheorem{corollary}[equation]{Corollary}
\newtheorem{lemma}[equation]{Lemma}
\newtheorem{conjecture}[equation]{Conjecture}
\newtheorem {problem}{Problem}
\newtheorem {remark}[equation]{Remark}

\theoremstyle{definition}
\newtheorem{definition}[equation]{Definition}

\newtheorem{example}[equation]{Example}

\def\a{\alpha}

\def\R{\mathbb R}
\def\N{\mathbb N}

\def\Z{\mathbb Z}

\def\cP{\mathcal{P}}

\def\Newt{\operatorname{Newt}}
\def\Ext{\operatorname{Ext}}
\def\conv{\operatorname{conv}}
\def\Supp{\operatorname{Supp}}
\def\LP{\operatorname{LP}}
\def\MLP{\operatorname{MLP}}
\definecolor{applegreen}{rgb}{0.55,0.71,0.0}

\definecolor{darkpink}{RGB}{245,75,200}

\makeatletter

\def\showlineno{line \the\inputlineno}

\def\multiset#1#2{\ensuremath{\left(\kern-.3em\left(\genfrac{}{}{0pt}{}{#1}{#2}\right)\kern-.3em\right)}}

\definecolor{linenogrey}{RGB}{150,150,150}

\def\Newt{\operatorname{Newt}}
\def\Ext{\operatorname{Ext}}

\title{IDP for 2-Partition Maximal Symmetric Polytopes}
\author{Su Ji Hong}
\address{Department of Applied Mathematics and Statistics, Johns Hopkins University} \email{shong75@jh.edu}
\author{George D. Nasr}
\address{Department of Mathematics, Augustana University} \email{george.nasr@augie.edu}
\date{}

\begin{document}

\maketitle

\begin{abstract}
We provide a framework for which one can approach showing the integer decomposition property for symmetric polytopes. We utilize this framework to prove a special case which we refer to as $2$-partition maximal polytopes in the case where it lies in a hyperplane of $\R^3$. Our method involves proving a special collection of polynomials have saturated Newton polytope.
\end{abstract}

\section{Introduction}
We say a convex polytope is a \textit{lattice polytope} whenever it is the convex hull of integer points. Lattice polytopes have been studied under many different settings, such as mathematical optimization \cite{t17}, projective toric varieties \cite{f93,c24}, and tropical geometry \cite{k21}. This work concerns that relating to the {integer decomposition property} (IDP). This property has been studied for a variety of reasons. For example, it often is intertwined with the study of Ehrhart polynomials for polytopes. 
% , which is intertwined in the study of Ehrhart polynomials.  

A common case of study is the {Newton polytope} of a polynomial. In this setting, one can ask if the corresponding polynomial has {saturated Newton polytope} (SNP). There is an established history of finding polynomials which are SNP---a survey can be found in \cite{mty19}. A classic example of this includes {Schur polynomials}, which will be relevant for this paper.

As is often the case, demonstrating that a polynomial has SNP in turn can be used to argue that the corresponding Newton polytope has IDP. This technique has been used more frequently in recent years---for instance, one can see \cite{hong20} for a study on Schur polynomials and inflated symmetric Grothendieck polynomials, which \cite{n23} generalized this and other results by defining what they call \textit{good symmetric functions}. Both cases leverage the SNP property of a polynomial to prove a corresponding Newton polytope has IDP. 

% example: Schur polynomial, SGP, iSGP, good polyonmials. 

The work of \cite{n23} involved adding together Schur functions, which has a symmetric Newton polytope. When we say a polytope $\cP\subseteq \R^n$ is \textit{symmetric}, we mean that for any for any permutation $\pi$ on $n$ elements, we have $\pi \cP=\cP$, where $\pi$ acts on the coordinates of points in $ \cP$.  This leads us to wonder the following. 

\begin{problem}\label{problem:main}
    Do all symmetric lattice polytopes have IDP? If not, which ones do?
\end{problem}

% \suji{Did we ever confirm that symmetric lattice polytope is the right term for the types of polytopes we want to look at? I see that we define it in section 2. Should we mention that we will define it later? 
% \GDN{No we never confirmed this. In the second sentence before problem 1, we actually say ``symmetric Newton polytope", so already there we are assuming some common definition of what symmetric means. Maybe we should, right after that, add ``, by which we mean..." and then add that definition from section 2? \suji{ I think we should do that.}}

% }
% study into related questions has already started. As one particular example, Oda conjectured every smooth lattice polytope has IDP \cite{o08}. A special case of this conjecture was proved by \cite{b19} in the case where the polytopes were centrally symmetric and three-dimensional. 

In this report, we focus on a special case of the aforementioned problem when the polytopes are what we call \textit{2-partition maximal polytopes}. For now, we wait to define these until the the next section. By using the aforementioned strategy, we identify the symmetric polynomials whose sum gives rise to 2-partition maximal symmetric polytopes and then prove the following. 

\begin{theorem}\label{thm:main}
If a 2-partition maximal symmetric polytope is contained in a hyperplane of $\R^3$, then it has IDP.
\end{theorem}

We will provide necessary terminology and background in section \ref{sec:prelim}. In section \ref{sec:method}, we provide a general framework by which one can prove symmetric polytopes are IDP. Finally, in section \ref{sec:thm}, we provide an outline for how we use the method from section \ref{sec:method} to prove Theorem \ref{thm:main}.

\section{Preliminaries}\label{sec:prelim}
Throughout we only work with \textit{lattice polytopes}, which we have defined as the convex hull of integer points. We additionally define the following things for polytopes. Throughout, $\cP$ is a polytope in $\R^n$.

\begin{definition}
\leavevmode
    \begin{itemize}
        \item The \textbf{dilation} of $\cP$ by $t\in \N$ is 
\(t\cP:=\{tx:x\in \cP\}.\)
\item  $\cP$ has the \textbf{integer decomposition property (IDP)} if for all $t\in \N$, $x\in t\cP$ implies $\displaystyle x=\sum_{i=1}^tx_i$ where $x_i\in \cP$. The $x_i$ need not be distinct.
\item The set $\conv(x_1,x_2,\dots, x_n)$ is the convex hull of points $x_1,x_2,\dots, x_n$.
\item Given a polynomial $\displaystyle f(x)=\sum_{\a} c_{\a}x^{\a} \in \R[x_1,x_2,\dots,x_n]$, the \textbf{support} of $f(x)$ is 
\[\Supp(f):=\{\a:c_\a\neq 0\}.\]
\item Given a polynomial $f(x) \in \R[x_1,x_2,\dots,x_n]$, the \textbf{Newton Polytope} is \[\Newt(f):=\conv(\Supp(f)).\]
\item     We say a polynomial $f$ has \textbf{saturated Newton polytope (SNP)} if every point in $\Newt(f)\cap \Z^n$ is an exponent vector in $f$.
% \item $\Ext(\cP)$ is the set of \textbf{extreme points}, or vertices, of $\cP$.
 
    \end{itemize}
\end{definition}

We will be interested in working with Schur polynomials and partitions, so we will additionally need the following. Throughout, $\lambda$ and $\mu$ are partitions.

\begin{definition}\leavevmode
\begin{itemize}
\item Let $T$ be a Young diagram for $\lambda$. A \textbf{semistandard
 Young tableau} $T$ is a filling with entries from $\Z_{>0}$ that is weakly increasing along
 rows and strictly increasing down columns. 
 \item The \textbf{content} of a semistandard Young tableau $T$ is $\alpha=(\alpha_1,\alpha_2,\dots, \alpha_n)$ where $\alpha_i$ is the number of $i$'s appearing in $T$.
    \item The \textbf{Schur polynomial} $s_\lambda\in \Z[x_1,x_2,\dots,x_n]$ is 
    \[\displaystyle s_\lambda=\sum_{\alpha} K_{\lambda,\alpha}x_1^{\alpha_1}x_2^{\alpha_2}\cdots x_n^{\alpha_n}\] where $K_{\lambda,\alpha}$ is the number of semistandard Young tableau of shape $\lambda$ and content $\alpha$. 
    \item If $\lambda$ and $\mu$ are partitions with the same size, then the \textbf{dominance order} $\trianglelefteq$, defined on partitions of the same size, is defined by $\lambda \trianglelefteq \mu$ whenever $\displaystyle \sum_{i=1}^k \lambda_i\leq \sum_{i=1}^k \mu_i$ for all $k$. In this case, we say $\lambda$ is \textbf{dominated} by $\mu$.
    % \item  $\lambda$ and $\mu$ are said to be \textbf{maximal} whenever we have $\lambda \not\trianglelefteq \mu$ and $\mu\not\trianglelefteq \lambda$. In general, a collection of partitions is maximal if every pair is maximal.
\item Given a symmetric polytope $\cP$, let $\LP(\cP)$ be the set of all lattice points of $\cP$ that are partitions and $\MLP(\cP)$ be the set of pairwise $\trianglelefteq$-maximal partitions in $\LP(\cP)$. 

\item For a polytope $\cP$, we let $\Ext(\cP)$ be the set of extreme points of $\cP$. 
\end{itemize}
\end{definition}

% {\color{red}
% Do we need the $\lambda_i$ to be the same size for the following to make sense? We want examples. No, it does make sense, we just need to include all the maximal partitions in the original polytope }

We now define the polytopes of interest to this paper.

%\GDN{SU JI tell me your thoughts on this. I felt we needed to give a name to the case we are studying.}
\begin{definition}
    A symmetric polytope $\mathcal{P}$ is \textit{2-partition maximal} provided there are two partitions $\lambda$ and $\mu$ for which $\MLP(\cP)=\{\lambda,\mu\}$.
    % \begin{enumerate}
    %     \item $\lambda$ and $\mu$ are maximal; and
    %     % \item The size of $\lambda$ is the same as the size of $\mu$; and
    %     \item If $\pi\in \mathcal{P}$, then $\pi$ is dominated by a permutation of $\lambda$ or $\mu$. That is, the only maximal points in $\cP$ are $\lambda$ and $\mu$.
    % \end{enumerate}
\end{definition}

% given $\mathcal{P}$, define $S_{\mathcal{P}}$

% \textbf{Example summary}
\begin{example}\label{ex:one}
    Consider a symmetric polytope $\cP$ whose extreme points are permutations of $(10,2,1)$ and $(7,6,0)$. See Figure \ref{fig:2-partition}. If one draws the integer points in $\cP$, as well as the Newton polytopes for $s_{(10,2,2)}$ and $s_{(7,6,0)}$, one can see that $\MLP(\cP)=\{(10,2,1),(7,7,0)\}$ as any integer point in the polytope is in one of the two Newton polytopes of the aforementioned Schur functions. See Figure \ref{fig:2-partition_lattice}. Thus the polytope is 2-partition maximal. 

    However, a symmetric polytope $\cP'$ whose extreme points are permutations of $(10,2,2)$ and $(7,7,0)$ contains a lattice point $(8,5,1)$. Indeed, observe that 
      \[(8,5,1)=\frac{7}{16}(10,2,2)+ \frac{1}{16}(2,10,2) +\frac{1}{2}(7,7,0).\]
      Furthermore, $\MLP(\cP')=\{(10,2,2), (7,7,0),  (8,5,1)\}$, so this polytope is not 2-partition maximal. See Figure \ref{fig:non-2-partition}.

\begin{figure}[h]
     \centering
     \begin{subfigure}[b]{0.3\textwidth}
         \centering
         \includegraphics[width=\textwidth]{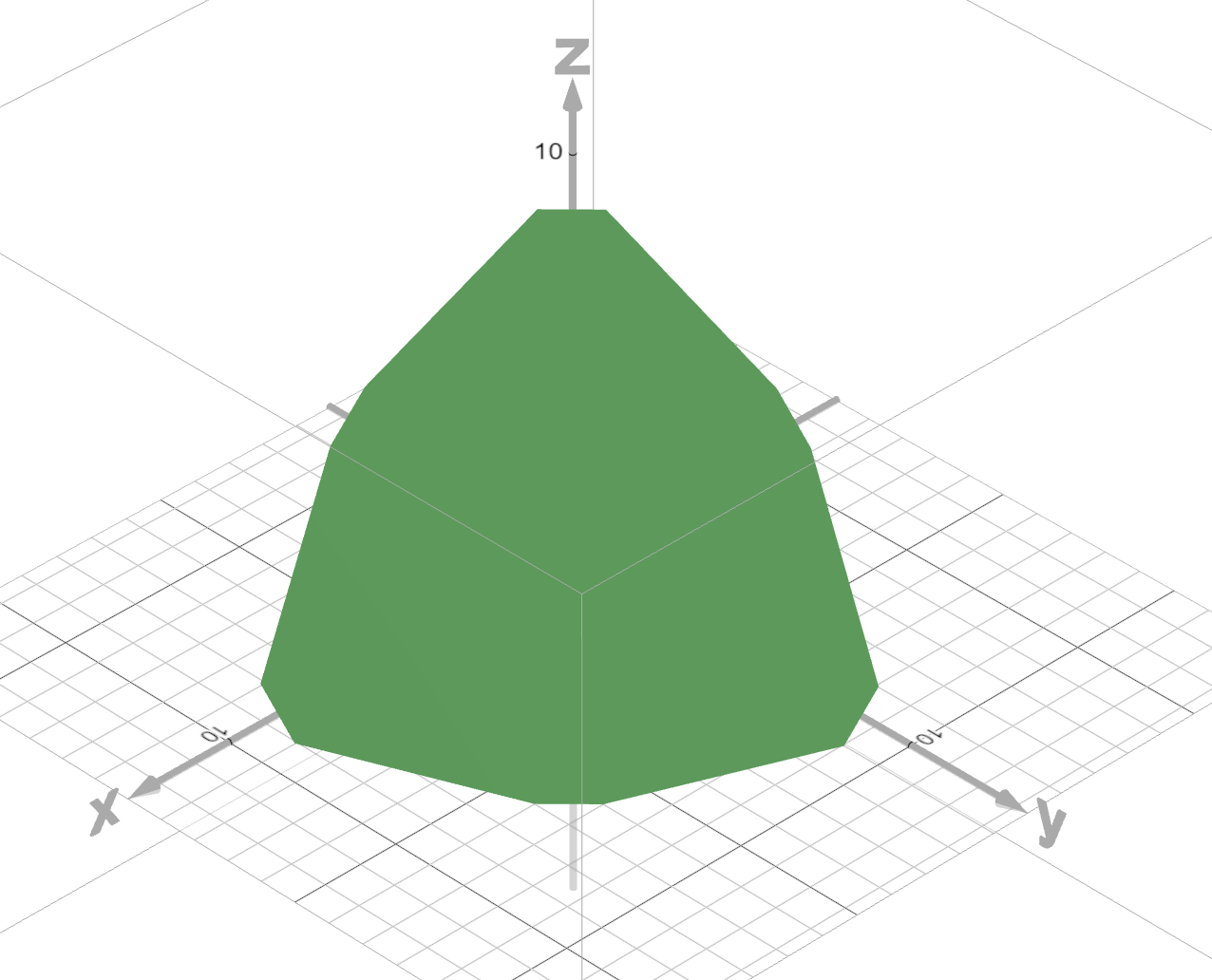}
         \caption{A symmetric polytope with extreme points of $(10,2,1)$ and $(7,6,0)$ and their permutations. }\label{fig:2-partition}
     \end{subfigure}
     \hfill
     \begin{subfigure}[b]{0.3\textwidth}
         \centering
         \includegraphics[width=\textwidth]{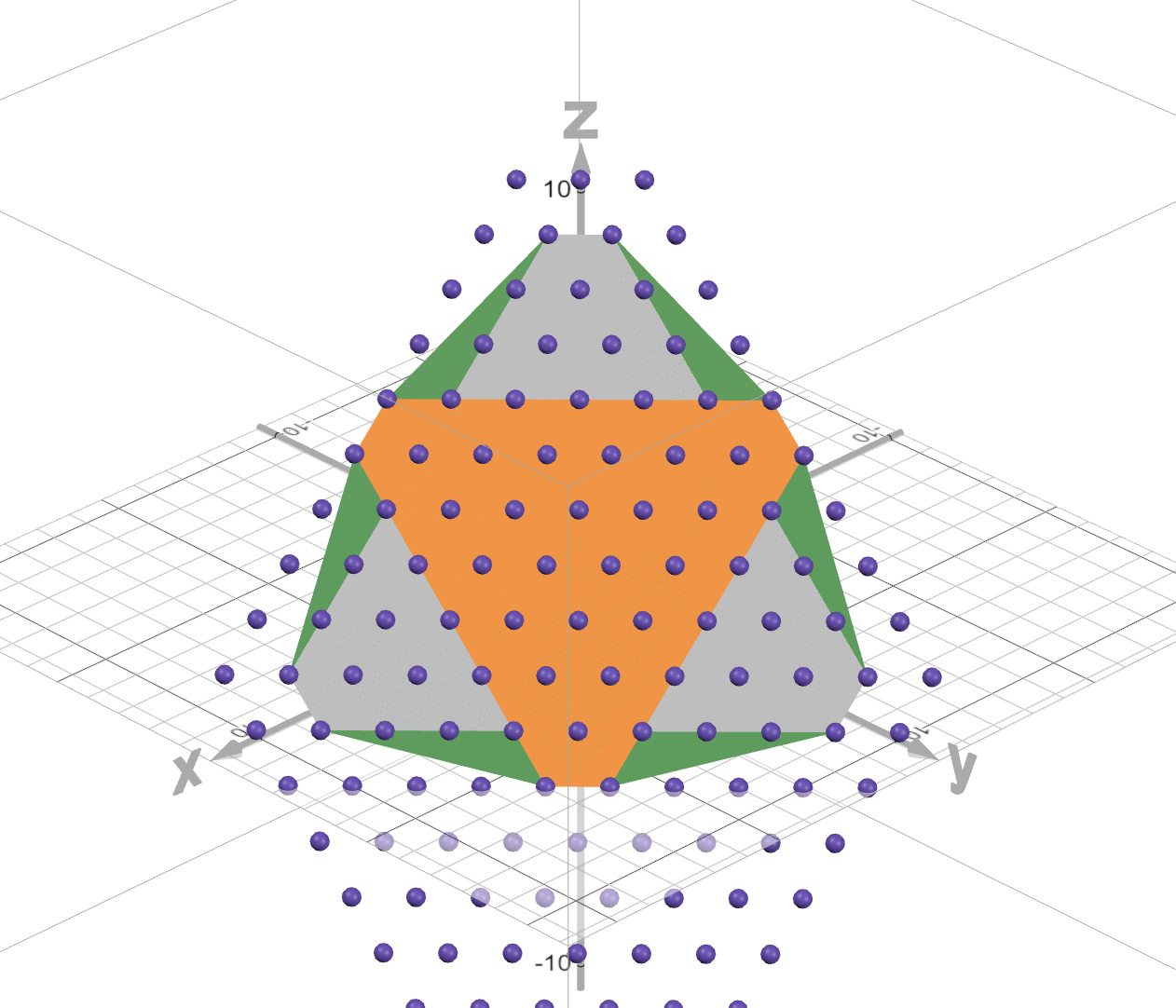}
         \caption{Figure \ref{fig:2-partition} with the Newton polytopes of $s_{(10,2,1)}$ and $s_{(7,6,0)}$ along with the integer points.}\label{fig:2-partition_lattice}
     \end{subfigure}
     \hfill
     \begin{subfigure}[b]{0.3\textwidth}
         \centering
         \includegraphics[width=\textwidth]{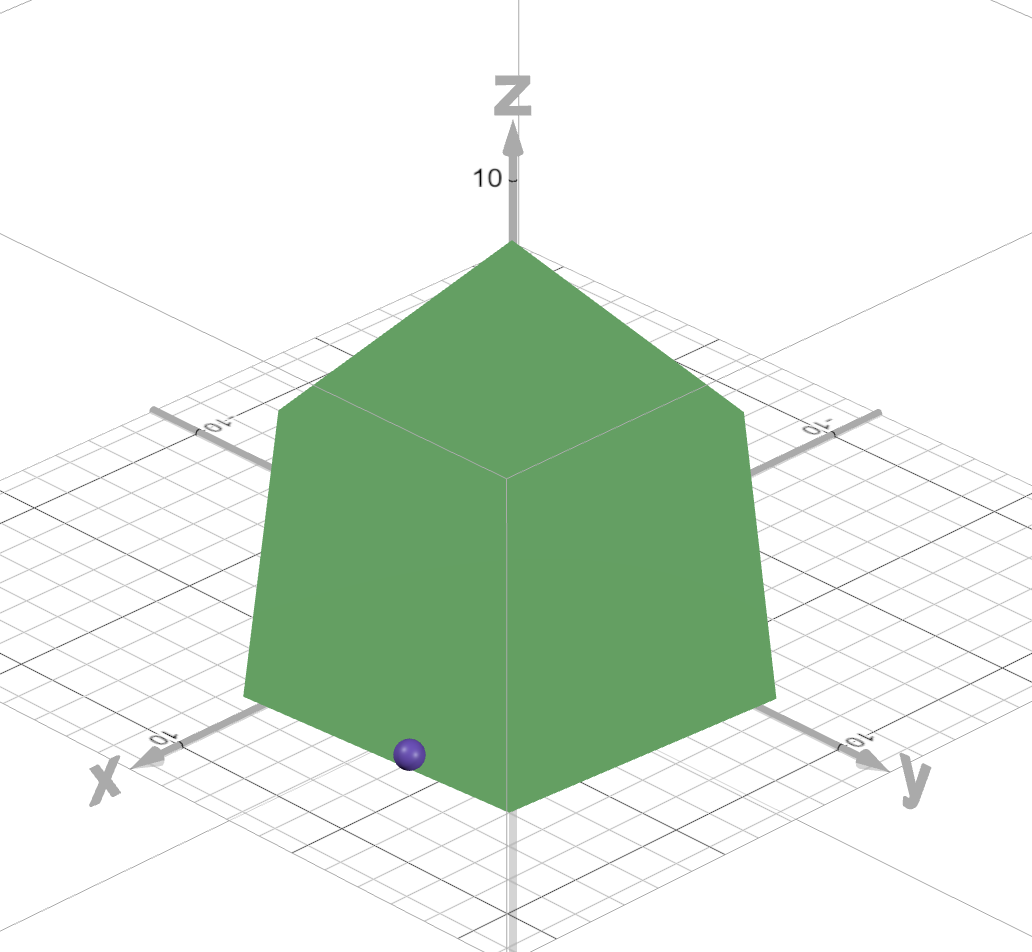}
         \caption{A symmetric polytope with extreme points of $(10,2,2)$ and $(7,7,0)$ and their permutations. The point $(8,5,1)$ lies in the polytope.}\label{fig:non-2-partition}
     \end{subfigure}
        \caption{}
        \label{fig:three graphs}
\end{figure}
\end{example}

\section{Conjecture on Symmetric Polytopes and IDP} \label{sec:method}
 We now discuss how one can extend ideas in \cite{hong20,n23} for our purposes.
% As mentioned in the introduction, there are partial answers to the Problem \ref{problem:main} in the literature.
The authors of \cite{hong20,n23} used certain families of polynomials $s$ which have SNP to show that $\Newt(s)$, which is symmetric, has IDP. To use these same strategies, we define $s$ in the following way.

\begin{lemma}
Let $\displaystyle s:= \sum_{\lambda\in \MLP(\cP)}s_\lambda$. Then $\cP = \Newt(s)$. \label{lem:newton_sum}
\end{lemma}

% This follows from the fact that $\lambda$ and $\mu$ make up the complete collection of pairwise-maximal partitions in $\cP$. This is something other authors take advantage of, as we expand on briefly in the next section.

% Consider instead $t\cP$. While it is true that $\Newt({s_{t\lambda}})\cup \Newt({s_{t\mu}})\subseteq t\cP$, the reverse containment is not generally true. This is to say that together, ${s_{t\lambda}}$ and ${s_{t\mu}}$ do not have enough information together to recover $t\cP$. To employ the methods from the introduction, we will need to find a polynomial whose Newton polytope is $t\cP$. To this end, we have the following. 

 The definition of convex hull and Rado's Theorem \cite{rado} provides justification for why this Lemma holds, which we detail now. Given a symmetric polytope $\cP$ in $\R^m$, if $\lambda$ is a lattice point in $\cP$, then so is $\pi(\lambda)$ for all permutations $\pi$ in $S_m$, the symmetric group on $m$ elements. Then by the definition of Newton polytopes, we see that $\Newt(s_\lambda) \subseteq \cP$. 
% Let $LP(\cP)$ be the set of all lattice points of $\cP$ that are partitions.
Thus, the support of $\displaystyle s=\sum_{\lambda\in \LP(\cP)} s_\lambda$ contains all lattice points of $\cP$, so $\cP=\Newt(s)$. However, by Rado's Theorem \cite{rado}, we see that $\lambda\trianglelefteq\mu$ if and only if $\Newt(s_\lambda) \subseteq \Newt(s_\mu)$. Thus, the support of $s$ is equal to the support of $\displaystyle \sum_{\lambda \in \MLP(\cP)} s_\lambda$.

Returning now to our situation, we see that if $\lambda$ and $\mu$ are the partitions for a 2-partition maximal polytope $\cP$ and $s:=s_\lambda+s_\mu$, then $\Newt(s)=\cP$. 
To employ the methods of \cite{hong20,n23}, we will want to find a polynomial with SNP which is $t\cP$ for each choice of $t$. To this end, we define the following.

\begin{definition}\label{definition:ts}
    Given $k$ partitions $\lambda_1,\lambda_2,\dots,\lambda_k$, not necessarily distinct, let  $s := s_{\lambda_1}+s_{\lambda_2}+\cdots+s_{\lambda_k}$. We define a new function $ts$, for $t\in \Z_{>0}$. Let $K:=\{1,2,\dots, k\}$ and $\multiset{K}{t}$ be the set of multisubsets of $K$ of size $t$. For $I\in \multiset{K}{t}$, let $\lambda_I:=\displaystyle \sum_{i\in I} \lambda_{i}$. Then
    % $tS = S_{\sum_{i=1}^k \lambda_{j_i}}$

    \[ts:=\displaystyle \sum_{I\in \multiset{K}{t}}s_{\lambda_I}.\]

\end{definition}

\begin{remark}
    
Note when $k=2$, $ts$  becomes 
\[\sum_{i=0}^t s_{(t-i)\lambda_1+i\lambda_2}.\]

\end{remark}

% where $MLP(\cP)$ is the set of pairwise maximal partitions in $LP(\cP)$. 

% In the prior section, we defined $ts$ as a means of extending the relationship that $s$ has with $\cP$.
In general, we conjecture the following for this special choice of $ts$ whenever $\cP$ is a symmetric polytope.
\begin{conjecture}\label{conjecture:snp}
    For $t\in \Z_{>0}$, we have $\Newt(ts)=t\Newt(s)$ and thus $t\cP=\Newt(ts)$. Furthermore, $ts$ has SNP for all $t$. 
\end{conjecture}
%The following is an outline of a general method which one may be able to use to demonstrate a symmetric polytope $\cP$ has IDP.
%\begin{enumerate}
%    \item Write down a sum $s:=\sum_{\lambda} s_\lambda$ of Schur polynomials so that the collection of partitions $\lambda$ are pairwise maximal and $\Newt(s)=\cP$.
 %   \item  For $t\in \N$, construct $ts$, which is a polynomial so that we have $\Newt(ts)=t\Newt(s)$ and $t\cP=\Newt(ts)$.
%\item Demonstrate $ts$ has SNP for all $t$.
%\end{enumerate}

If this conjecture is true, then the set of lattice points of $t\cP$ is the support of $ts$. Then we can apply the following.

\begin{theorem}\label{thm:IDP}
    Let $\cP$ be a symmetric polytope for which the prior conjecture holds true for $\displaystyle s= \sum_{\lambda\in \MLP(\cP)} s_\lambda$. (That is, $s$ has SNP.) Then $\cP$ has the integer decomposition property. 
\end{theorem}

\begin{proof}
    Let $\cP$ be a symmetric polytope for which the prior conjecture holds true. Let $t$ be a positive integer. We want to show that every lattice points of $t\cP$ can be written as sum of $t$ lattice points of $\cP$. Let $\pi\in t\cP$. Since $ts$ has SNP, then $\pi$ is an exponent vector of 
        \[ts=\displaystyle \sum_{I\in \multiset{K}{t}}s_{\lambda_I},\]
that is, there is an $I:=\{i_1,i_2,\dots, i_t\}$ so that $\pi$ is an exponent vector of $s_{\lambda_{i_1}+\lambda_{i_2}+\cdots +\lambda_{i_t}}$. Each exponent vector of $s_{\lambda_{i_1}+\lambda_{i_2}+\cdots +\lambda_{i_t}}$ corresponds to a content vector of some semistandard tableau, which we denote $T$, for the partition $\lambda_{i_1}+\lambda_{i_2}+\cdots +\lambda_{i_t}$. We demonstrate that any such tableau can be decomposed into semistandard tableaux for $\lambda_{i_1}$, $\lambda_{i_2}$, $\dots$, $\lambda_{i_t}$, which we will denote $T_{i_j}$ for $j$ from $1$ to $t$ respectively. We define these using an inductive procedure, focusing on one row of each tableau at a time, starting with the last row for each.

 Let $n$ be the number of rows in $T$ and let $(\lambda_{i_j})_k$ denote the $k$th part of $\lambda_{i_j}$.
For each $j$, starting at $j=1$ do the following.
\begin{enumerate}
    \item Let $\ell$ be the number of rows of $T$. Remove the first $(\lambda_{i_j})_\ell$ columns from $T$ and use them as the right-most columns of $T_{i_j}$. Note that these columns all have length $\ell$.
\item Subtract $(\lambda_{i_j})_\ell$ from each entry of $\lambda_{i_j}$, truncating the last coordinate.
\end{enumerate}
We now can repeat these steps by inducting on the length of the partitions to find the remaining columns for $T_{i_j}$.

Each $T_{i_j}$ is a semistandard tableaux since columns of $T_{i_j}$ come directly from columns of $T$ and they are added from left-to-right.
Each $T_{i_j}$ is also of shape $\lambda_{i_j}$. We prove this by inducting on the rows of $T_{i_j}$. Observe by step $1$ of the above procedure that this is true for row $n$. After the first iteration of the two steps above is completed, note that the length of the partitions decreases. In general, suppose row $r$ of $T_{i_j}$ has $(\lambda_{i_j})_r$ cells. Then the above procedure will add $(\lambda_{i_j})_{r-1}-(\lambda_{i_j})_{r}$ columns from $T$ to $T_{i_j}$, meaning that row $r-1$ will have $(\lambda_{i_j})_{r-1}-(\lambda_{i_j})_{r}+(\lambda_{i_j})_{r}=(\lambda_{i_j})_{r-1}$ cells. Furthermore, after the above steps are completed for row $r-1$, $T$ will have length $r-2$, meaning that no further cells will be added to row $r-1$.

% \begin{enumerate}
%     % \item Make all partitions $\lambda_{i_1}$, $\lambda_{i_2}$, $\dots$, $\lambda_{i_t}$ the same length by adding $0$s to the end of them if needed. 
%     % \item Reindex $\lambda_{i_j}$  so that the partitions with the largest part appear first. 
%     \item The  first $(\lambda_{i_j})_\ell $  columns of  $T_{i_j}$  are the columns $(\lambda_{i_{1}})_\ell+\cdots +(\lambda_{i_{j-1}})_\ell+1$ through $(\lambda_{i_{1}})_\ell+\cdots +(\lambda_{i_{j}})_\ell$ of $T$. (If $j-1=0$, then the first $\lambda_{j}$ columns of $T$ are assigned to $T_{i_1}$.)
%     \item Now remove $(\lambda_{i_{j}})_\ell$ from each coordinate of $\lambda_{i_j}$.
%     \end{enumerate}

% Let $\mu = \lambda_{i_1}+\lambda_{i_2}+\cdots +\lambda_{i_t} $. Then 

% \[(\mu)_k =\sum_{j=1}^t (\lambda_{i_j})_{k} \]

%     \[ \a_m = \sum_{k=0}^m  (\mu)_{\ell-k}= \sum_{k=\ell-m+1}^\ell  (\mu)_{k}=\sum_{k=0}^m (T)_{\ell-k}\]

% $T_{i_j} $: $\sum$

%     This makes it so that the number of cells in last row of $T_{i_j}$ is $\lambda_{i_j}$. 
    
%     We now inductively apply this process so that the other rows of $T_{i_j}$ have length corresponding to the other parts in $\lambda_{i_j}$. In general, assuming that the $m+1$th row of $T_{i_j}$ has length equal to $m+1$th part in $\lambda_{i_j}$, repeat steps $3$ and $4$ using $m$ instead of $\ell$.
% \GDN{Do I need to do step 2 during the inductive step? I wonder if changing the order is actually necessary. }

Consequently, each $T_{i_j}$ has a shape that corresponds to $\lambda_{i_j}$, which means the content vector of $T_{i_j}$ is a lattice point of $\cP$ since $s$ has SNP. Thus $\pi$, which is the content vector of $T$, can be written as a sum of $t$ lattice points of $\cP$. This implies $\cP$ has IDP.

%$\pi \in \Supp(s_{\lambda_{i_1}+\lambda_{i_2}+\cdots +\lambda_{i_t}}) = \Supp(s_{\lambda_{i_1}})\cup\Supp(s_{\lambda_{i_2}})\cup\cdots \cup \Supp(s_{\lambda_{i_t}})$

\end{proof}

\begin{example}

In Figure \ref{fig:tab_decomp_ex} is a semistandard tableau for the partition $(8,7,2)$, which can be decomposed into the three tableaux with shapes $\lambda=(2,2,1)$, $\mu=(4,3,1)$ and $\gamma=(2,2,0)$ using the algorithm in the prior proof.

Since the last part of $\lambda$ and $\mu$ is $1$, we start by designating the first column for $T_\lambda$ and the second for $T_\mu$ (depicted in Figure \ref{fig:tab_decomp_ex_first_step}).

\begin{figure}[h]
\ytableausetup{centertableaux}
     \begin{subfigure}[b]{0.3\textwidth}
         \centering
         \begin{ytableau}
 1 &  1 &  1 &2 &2 & 3 &  3 &  3 \\
2 &  2 &  3 &   3 &   4 &   4 & 4 \\
 3 &    4 
\end{ytableau}
         \caption{}\label{fig:tab_decomp_ex}
     \end{subfigure}
     \hfill
     \begin{subfigure}[b]{0.3\textwidth}
         \centering
         \begin{ytableau}
*(red!50) 1 & \none & *(blue!50)  1 &   \none & 1 &2 & 2 & 3 & 3 &  3 \\
*(red!50)2 & \none & *(blue!50) 2 & \none & 3 & 3 & 4 &4 &  4 \\
*(red!50) 3 & \none &  *(blue!50)  4 
\end{ytableau}
         \caption{}\label{fig:tab_decomp_ex_first_step}
     \end{subfigure}
% %      \hfill
% %      \begin{subfigure}[b]{0.1\textwidth}
% %          \centering
% %          \begin{ytableau}
% % *(red!50) 1   \\
% % *(red!50) 2  \\
% % *(red!50) 3 
% % \end{ytableau}
% %          \caption{}\label{fig:tab_decomp_ex_lambda1}
% %      \end{subfigure}
% %      \hfill
% %      \begin{subfigure}[b]{0.1\textwidth}
% %          \centering
% % \begin{ytableau}
% % *(blue!50) 1  \\
% % *(blue!50) 2 \\
% % *(blue!50) 4
% % \end{ytableau}
%          \caption{}\label{fig:tab_decomp_ex_mu1}
%      \end{subfigure}
        \caption{The first step of the decomposition process for a semistandard tableau of shape $(8,7,2)$.}
        \label{fig:three graphs}
\end{figure}

Modifying our partitions, we now have $\lambda=(1,1,0)$, $\mu=(3,2,0)$, and $\gamma=(2,2,0)$. Continuing with the updated version of $T$, the first column is assigned to $T_\lambda $, the following two are assigned to $T_\mu$, and the following two columns are assigned to $T_\gamma$. At the end of this step, $T_\lambda$ and $T_\gamma$ are complete (see Figures \ref{fig:tab_decomp_ex_lambda2} and \ref{fig:tab_decomp_ex_gamma}, respectively), and all but the final column of $T$ has been assigned. This final column will be assigned to $T_\mu$. See Figure \ref{fig:tab_decomp_ex_secondthirdstep} below for the identification of these columns in $T$ and the updated picture for $T_\gamma$ in Figure \ref{fig:tab_decomp_ex_mu2}.

\begin{figure}[h]
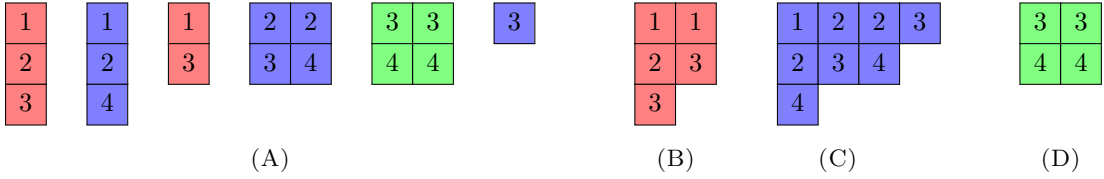

\ytableausetup{centertableaux}
     \begin{subfigure}[b]{0.5\textwidth}
         \centering
 \begin{ytableau}
*(red!50) 1 & \none & *(blue!50)  1 &\none &  *(red!50)1 &\none& *(blue!50) 2 &  *(blue!50) 2 &\none & *(green!50) 3 &*(green!50) 3 &  \none & *(blue!50) 3\\ 
*(red!50) 2 & \none & *(blue!50)  2 &\none &  *(red!50)3 &\none& *(blue!50) 3&  *(blue!50) 4 &\none & *(green!50) 4 &*(green!50) 4 &  \none & \none\\ 
*(red!50) 3 & \none & *(blue!50)  4 \\
\end{ytableau}
                 \caption{}\label{fig:tab_decomp_ex_secondthirdstep}
     \end{subfigure}
     \hfill    
     \begin{subfigure}[b]{0.1\textwidth}
         \centering
      \begin{ytableau}
*(red!50) 1 &*(red!50)  1  \\
*(red!50) 2 &*(red!50)  3  \\
*(red!50) 3 
\end{ytableau}
            \caption{}\label{fig:tab_decomp_ex_lambda2}
     \end{subfigure}
     \hfill
     \begin{subfigure}[b]{0.1\textwidth}
         \centering
         
         \begin{ytableau}
*(blue!50) 1 &*(blue!50)  2 &*(blue!50) 2&*(blue!50)  3 \\
*(blue!50) 2 &*(blue!50) 3&*(blue!50)  4  \\
*(blue!50) 4
\end{ytableau}
         \caption{}\label{fig:tab_decomp_ex_mu2}
     \end{subfigure}
     \hfill
     \begin{subfigure}[b]{0.2\textwidth}
         \centering
         \begin{ytableau}
*(green!50)   3 &*(green!50)   3 \\
*(green!50)  4 &*(green!50)   4 \\
\none
\end{ytableau}

         \caption{}\label{fig:tab_decomp_ex_gamma}
     \end{subfigure}
        \caption{Steps two and three of the decomposition process of a tableau of shape $(8,7,2)$.}
        \label{fig:three graphs}
\end{figure}

\end{example}

As an example of an application for the prior Theorem, if $\cP$ is a symmetric polytope with two maximal lattice partitions so that Conjecture \ref{conjecture:snp} holds, then each lattice points of $t\cP$ is a content vector of a semistandard Young tableau of $(t-k)\lambda + k\mu$ for $0\leq k\leq t$, which can be separated into $t-k$ content vector for $\lambda$ and $k$ for $\mu$. Note that the content vectors of $\lambda$ and $\mu$ are themselves lattice points of $\cP$. Thus the lattice point can be written as a sum of $t$ points from $\cP$.

% {\color{green}
% Su Ji conjecture: If size of two partitions differ by more than one then there are other maximal points}

\begin{remark}
One can apply the above framework to aforementioned results. For example, consider the work done by \cite{hong20}. There, they define  the symmetric Grothendieck polynomial $G_\lambda$ of a partition $\lambda$ and study $\cP=\Newt(G_\lambda)$. The polynomial $G_\lambda$ is of the form
\[G_\lambda=\sum_\mu c_\mu s_\mu\]
where $c_\mu\in \Z$ and the sum is over partitions $\mu$ obtained by adding boxes to the Young diagram of $\lambda$. The partitions $\mu$ in the above descriptions are not all maximal partitions. However, in \cite{ey17}, they have shown that the $\cP$ is the convex hull of the support of $\displaystyle\sum_{\mu \in \MLP} s_\mu$. 
In \cite{hong20}, the authors showed that $\cP$ has IDP by observing that $t\cP = \Newt(G_{t,t\lambda})$ where $t\in \Z_{>0}$ and $G_{t,t\lambda}$ is the inflated symmetric Grothendieck polynomial. In addition, they showed that the $t\cP$ is the convex hull of the support of $\displaystyle \sum_{\mu \in \MLP(t\cP)} s_\mu$. Thus, one can define $s=\displaystyle \sum_\mu s_\mu$ (with maximal lattice partition $\mu$) and apply the method outlined in this section to give an alternative proof that $\cP$ is IDP.

\end{remark} 

% \section{The Method}
% The following is an outline of how one can demonstrate a symmetric polytope has IDP. 
% \begin{enumerate}
%     \item Start with symmetric polytope $P$
%     \item Write down sum S of Schur polynomials corresponding to maximal partitions
%     \item Construct $tS$ by combining $t$ multisets of partitions in $S$
% \item $tS$ has SNP
% \item $tP$ is Newton(tS) for all t
% \item $P$ is IDP 
% \end{enumerate}

% We will show this method holds so long as step 4 is known to be true. 

\section{2-Partition Maximal Polytopes and IDP}\label{sec:thm}
% Goal: Let $\mathcal{P}$ be a ``symmetric'' polytope where there are only two maximal partition $\mu$ and $\lambda$. Let $S = S_\lambda+S_\mu$. show that $tS$ has SNP in the case of $k=2$. 

In this section we apply the method outlined in the prior section in our special case. Throughout, $\lambda$ and $\mu$ are pairwise maximal partitions corresponding to a 2-partition maximal polytope $\cP$.
We will focus on the case where the size of $\lambda$ and $\mu$ are the same, leaving the other cases for future study. 
Thus, $\cP$ lies in a hyperplane. We further restrict our study to the case where $\cP\subseteq \R^3$, meaning that the length of $\lambda$ and $\mu$ is 3, as in this case we learn even more about the partitions $\lambda$ and $\mu$. 
% \begin{proposition}%Proposition?
%     $\lambda$ and $\mu$ are partitions with sizes $n_\lambda$ and $n_\mu$. Then $|n_\lambda-n_\mu| \leq 1$.  
% \end{proposition}

% \begin{remark}
%     One should note that the converse of the prior proposition is not true. For example, see Figure \ref{fig:non-2-partition} and Example \ref{ex:one}.

% \end{remark}

\begin{lemma} \label{lem:same_size_differ_1}

    If the lengths of $\lambda$ and $\mu$ are 3 and both partitions have the same size, then 
    $|\lambda_i-\mu_i|= 1$ for some $i\in \{1,2,3\}$. 

\end{lemma}

\begin{proof}
    Suppose that we have $\lambda=(\lambda_1,\lambda_2,\lambda_3)$ and $\mu=(\mu_1,\mu_2,\mu_3)$ so that $|\lambda_i-\mu_i|\geq 2$ for all $i$. Without loss of generality, suppose $\lambda_1>\mu_1$. Thus, for $\lambda$ and $\mu$ to be maximal (and not comparable), we must have $\lambda_3>\mu_3$. Specifically, note that $\lambda_3>\mu_3+1$. 
    
    Now, let $\gamma:=(\mu_1+1,n-\mu_1-\lambda_3,\lambda_3-1)$. Note that $\gamma$, $\lambda$, and $\mu$ are pairwise maximal. We prove that $\gamma\in \cP$.

    %\suji{Since $\lambda$ and $\mu$ are not comparable, $\lambda_1\neq \mu_1$ and $\lambda_3\neq \mu_3$. Otherwise we are comparing just two entries in $\lambda$ and $\mu$. The partitions of length 2 are always comparable. Thus $\lambda_1\neq \mu_1$. Assume $\lambda_1>\mu_1$.\\ Since $\lambda_1>\mu_1$ and $\lambda_1+\lambda_2+\lambda_3 = \mu_1+\mu_2+\mu_3$, in order for $\lambda$ and $\mu$ to be not comparable, $\lambda_1+\lambda_2 < \mu_1+\mu_2$. Thus $n-\lambda_3< n-\mu_3$, i.e., $\lambda_3>\mu_3$.} 

First, consider the point $\iota:=(\mu_1,n-\mu_1-\lambda_3,\lambda_3)$. Observe that $\iota\in \cP$ since $\mu$ and $\lambda$ dominate $\iota$. Thus, it is sufficient to show that $\gamma\in \conv(\lambda, \mu,\iota)$. To this end, we define a matrix $A=\begin{bmatrix}
    \lambda & \mu & \iota
\end{bmatrix}$, that is, a $3\times 3$ matrix whose columns are $\lambda$, $\mu$, and $\iota$. We will demonstrate that $A^{-1}$ exists and $A^{-1}\gamma$ has non-negative entries which sum to $1$. 

To proceed, we use  Sage's Matrix class to perform the necessary computations. These can be seen in Appendix \ref{app:sage}.
    First, the determinant of $A$ is 
    \[\big(\lambda_3(\mu_1 +\mu_2) -\mu_3(\lambda_1+\lambda_2)\big)(\lambda_1-\mu_1)=n(\lambda_3 -\mu_3)(\lambda_1-\mu_1)\]
    which is nonzero since $\lambda_1>\mu_1$ and $\lambda_3>\mu_3$, and so $A^{-1}$ exists.

    We next demonstrate that the entries of $A^{-1}\gamma$ sum to $1$. After computation, we find that
    \[ A^{-1}\gamma= \begin{bmatrix}
        {  1\over (\lambda_1-\mu_1)} & {1\over (\lambda_3-\mu_3)} & 1+{\lambda_2-\mu_2 \over (\lambda_1-\mu_1)(\lambda_3-\mu_3)}
    \end{bmatrix}.\]

    If one adds the first two entries together, utilizing the fact that the size of $\lambda$ and $\mu$ are the same, we have 

    \[{  1\over (\lambda_1-\mu_1)}+ {1\over (\lambda_3-\mu_3)} ={\lambda_1+\lambda_3-(\mu_1+\mu_3)\over (\lambda_1-\mu_1)(\lambda_3-\mu_3)}={\mu_2-\lambda_2\over (\lambda_1-\mu_1)(\lambda_3-\mu_3)},\]
    which when added the final entry of $A^{-1}\gamma$ gives $1$. Now, since $\lambda_1>\mu_1$ and $\lambda_3>\mu_3$, one can see the first two entries of $\gamma$ are positive. Since the last entry is $1$ minus the first two entries of $\gamma$, both of which are at most ${1\over 2}$ due to our assumption that $|\lambda_i-\mu_i|\geq 2$, this means the final entry is non-negative. That is, we have demonstrated that $\gamma\in \conv(\lambda,\mu, \iota)\subseteq \cP$. This means that $\{\lambda,\mu,\gamma\}\subseteq \MLP(\cP)$, but we assumed that $\cP$ is $2$-partition-maximal. This is a contradiction.

    Thus, we know there exists an $i$ for which $|\lambda_i-\mu_i|\leq 1$. However, we now further that this value can not be $0$. However, since the length of $\mu$ and $\lambda$ are 3, if one of their coordinates agree, then after removing these coordinates we are really considering  partitions of length 2 with the same size. These partitions are always comparable, contradicting that our partitions are pair-wise not comparable.

\end{proof}

In the prior proof, we utilized the region $\conv(\lambda,\mu,\iota)$ with $\iota:=(\mu_1,n-\mu_1-\lambda_3,\lambda_3)$. This region has further utility. 

\begin{lemma}
    The interior of $\conv(\lambda,\mu,\iota)$ contains all weakly decreasing lattice points in $\cP$ which are not contained in $\Newt(s_\lambda)$ and $\Newt(s_\mu)$. \label{lem:dec}
\end{lemma}
\begin{proof}
For convenience we define  $N_\lambda:=\Newt(s_\lambda)$ and $N_\mu:=\Newt(s_\mu)$.

First, we demonstrate $\Ext(\cP)=\Ext(N_\lambda)\cup \Ext(N_\mu)$.
By Lemma \ref{lem:newton_sum}, $\cP$ is the convex hull of the exponent vectors for $s_\lambda$ and $s_\mu$. Without loss of generality, suppose $\pi\in \Ext(\cP)$ is an exponent vector of $s_\lambda$. If $\pi$ were not an extreme point of $N_\lambda$, then it is either on the boundary or interior of $N_\lambda$. Then the same would be true for $\cP$, which is a contradiction. 
On the other hand, suppose $\pi\in \Ext(N_\lambda)\cup \Ext(N_\mu)$ and yet $\pi \notin \Ext(N_s)$. Without loss of generality, $\pi \in \Ext(N_\mu)$, which implies that $\pi$ is one of the permutations of $\mu$. By definition of $N_s$, if $\pi \not \in \Ext(N_s)$, then none of the other permutations of $\pi$ is an extreme point of $N_s$.Then $\Ext(N_s) \subseteq \Ext(N_\lambda)$. This implies that $N_s \subseteq N_\lambda$. Thus $\mu \in N_s $ is a lattice point in $N_\lambda$, which means that $\mu$ is dominated by $\lambda$. However, this contradicts that $\mu$ and $\lambda$ are not comparable.  

%By construction of $N_s$, this suggests that no permutation of $\pi$ is in $N_s$ either. By choice of $\pi$, $\pi$ is, without loss of generality, one of the permutations of $\mu$. So this implies that $\Ext(N_s)=\Ext(N_\lambda)$ which implies $\mu$ is dominated by $\lambda$, but this contradicts our assumption on the relation of the permuations. 

Next, we demonstrate that $\iota=(\mu_1,n-\mu_1-\lambda_3,\lambda_3)$ is both on a facet of $N_\lambda$ and $N_\mu$. Using \cite[Theorem 29]{hong20}, note that $\iota$ is on the face defined by $x+y=\lambda_1+\lambda_2$, and so $\iota$ is on a facet of $N_\lambda$, and $\iota$ is on the face defined by $x=\mu_1$, so $\iota$ is also on a facet of $N_\mu$.

The line $(\lambda_3-\mu_3)x-(\lambda_1-\mu_1)z=\mu_1\lambda_3-\mu_3\lambda_1$ has both $\lambda$ and $\mu$ on it. We claim this line is a facet of $\cP$, that is, $(x,y,z)\in \cP$ implies $(\lambda_3-\mu_3)x-(\lambda_1-\mu_1)z\leq \mu_1\lambda_3-\mu_3\lambda_1$. Note that for all permutations $(x,y,z)$ of $\lambda$, $x\leq \lambda_1$ and $z\geq \lambda_3$. Thus
\begin{align*}
    (\lambda_3-\mu_3)x - ( \lambda_1 -\mu_1)z \leq& (\lambda_3-\mu_3)\lambda_1 - ( \lambda_1 -\mu_1)\lambda_3 \\=& \lambda_3 \mu_1 - \lambda_1\mu_3.
\end{align*}
Note that for all permutations $(x,y,z)$ of $\mu$, $x\leq \mu_1$ and $z\geq \mu_3$. Thus
\begin{align*}
    (\lambda_3-\mu_3)x - ( \lambda_1 -\mu_1)z \leq& (\lambda_3-\mu_3)\mu_1 - ( \lambda_1 -\mu_1)\mu_3 \\=& \lambda_3 \mu_1 - \lambda_1\mu_3.
\end{align*}
We have demonstrated that the extreme points of $\cP$ satisfy the aforementioned inequality, and thus every point in $\cP$ does as well. 

We now can finally demonstrate that the decreasing points of $\cP$ are contained in $\conv(\lambda,\mu,\iota)$. We first identify the facet defining inequalities of $\conv(\lambda,\mu,\iota)$. There are three facets: hyperplanes containing $\mu$ and $\lambda$, $\iota$ and $\mu$, and $\iota$ and $\lambda$. That is, the points of $\conv(\lambda,\mu,\iota)$ satisfy the following three inequalities:
\[(\lambda_3-\mu_3)x-(\lambda_1-\mu_1)z\leq \mu_1\lambda_3-\mu_3\lambda_1,\,\,\, x\geq \mu_1,\,\,\, \text{and } z \leq \lambda_3 .\]

Let $(x,y,z)\in \cP\setminus (N_\lambda\cup N_\mu)$ and $x\geq y\geq z$. Since $(x,y,z)\not\in N_\lambda,$ it must not satisfy at least one of the facet defining inequalities from \cite[Theorem 29]{hong20}: 
\begin{align*}
    x\leq \lambda_1, &\hspace{.3cm} x\geq \lambda_3,\\
    y\leq \lambda_1, &\hspace{.3cm} y\geq \lambda_3,\\
    z\leq \lambda_1, &\hspace{.3cm} z\geq \lambda_3.
\end{align*}
Since $x\geq y\geq z$, one of $x\leq \lambda_1$ or $z\geq \lambda_3$ must not be satisfied. However, $(x,y,z)\in \cP$, which is the convexhull of permutations of $\lambda$ and $\mu$. Since the coordinates of $\lambda$ and $\mu$ are at most $\lambda_1$, $x$, $y$, and $z$ cannot be more than $\lambda_1$. Thus $x\leq \lambda_1$ is always satisfied. Further, $z< \lambda_3$. 

Similarly, since $(x,y,z)\not\in N_\mu,$ $x\leq \mu_1$ or $z\geq \mu_3$ must not be satisfied. But as before, since the coordinates of $\lambda$ and $\mu$ are at least $\mu_3,$ we know that $z\geq \mu_3$. Thus $x>\mu_1$. Thus, we have demonstrated that any decreasing point in $\cP$ satisfies the three aforementioned facet defining inequalities. 
\end{proof}
The utility of this result is two-fold. First, recall that ultimately our goal is to show that if $s=s_\lambda+s_\mu$, we have $ts$ is SNP for all $t\in \N$. Since it is already known that $s_\lambda$ and $s_\mu$ are SNP, we need only be concerned with points outside the Newton polytope for these two points. 
Beyond this, we also know that a point $\pi\in \cP$ if and only if a permutation of $\pi$ is in $\cP$. Thus, we may only concern ourselves with decreasing points outside of the Newton polytopes of $s_\lambda$ and $s_\mu$, that is, the interior of the region $\conv(\lambda,\mu,\iota)$. With the prior Lemma in mind, we have the following from which Theorem \ref{thm:main} follows.

\begin{theorem}
    The polynomial 
    \[ts=\sum_{i=0}^t s_{(t-i)\lambda+i\mu}\]
has SNP. \label{thm:snp}
\end{theorem}
\begin{proof}
We show $(x,y,z)\in \Newt(s_{(t-i)\lambda+i\mu})$ for some $i$. By Lemma \ref{lem:same_size_differ_1}, we know there exists a coordinate between $\lambda$ and $\mu$ which differs exactly by $1$.  Without loss of generality, suppose $\lambda_3-\mu_3=1$. Then $t\mu_3\leq z\leq t\lambda_3$. Let $\alpha:=(t-(t\lambda_3-z))\lambda+(t\lambda_3-z)\mu$.
We claim that $(x,y,z)\in \Newt(s_{\alpha})$.

Observe the last coordinate of $\a$ is $z$, since 
\begin{align*}
    (t-(t\lambda_3-z))\lambda_3+(t\lambda_3-z)\mu_3 = &t\lambda_3 - t\lambda_3\lambda_3 +z \lambda_3  + t\lambda_3\mu_3 -z\mu_3\\
    =& t\lambda_3 +z(\lambda_3- \mu_3)-t \lambda_3 (\lambda_3  -\mu_3)
    \\=& z,
\end{align*}
where the last equality follows after implementing our assumption that $\lambda_3-\mu_3=1$.

By assuming $(x,y,z)$ is weakly decreasing, recall by Lemma \ref{lem:dec} we know that $(x,y,z)\in \conv(\lambda,\mu,\iota)$ where $\iota:=(\mu_1,n-\mu_1-\lambda_3,\lambda_3)$. Thus, we may write $(x,y,z) = at\lambda + bt\mu + c t\iota$ with $a,b,c>0$ and $a+b+c=1$. Since the last coordinates of $(x,y,z)$ and $\alpha$ are equal, we just need to show that $x$ is less than the first coordinate of $\a$, that is,
\[at\lambda_1 + bt\mu_1 + ct \mu_1 < (t-(t\lambda_3-z))\lambda_1+(t\lambda_3-z)\mu_1 = t\lambda_1 - (t\lambda_3-z)(\lambda_1 - \mu_1).\]

Since $z= at\lambda_3 + bt\mu_3 + ct\lambda_3$, we have 
\[t\lambda_3 -z = t\lambda_3-(at\lambda_3 + bt\mu_3 + ct\lambda_3) = t((1-a-c)\lambda_3-b\mu_3)= t(b\lambda_3-b\mu_3)=tb(\lambda_3-\mu_3) = tb .\]

Thus,
\begin{align*}
t\lambda_1 - (t\lambda_3-z)(\lambda_1 - \mu_1) =& t\lambda_1 - tb(\lambda_1-\mu_1)\\
=& t(\lambda_1 - b\lambda_1+b\mu_1)\\
=& t((1-b)\lambda_1+b\mu_1).
\end{align*}

Note that since $\mu_1 <\lambda_1$, we have

\begin{align*}
    x=at\lambda_1 + bt\mu_1 + ct \mu_1 <& at\lambda_1 + bt\mu_1 + ct \lambda_1\\
    = & t(b\mu_1 +(a+c) \lambda_1) \\
    =&t(b\mu_1 +(1-b) \lambda_1) = t\lambda_1 -(t\lambda_3 -z)(\lambda_1-\mu_1).
\end{align*}

That is, we have now shown that $(x,y,z)$ is dominated by $\a$.
\end{proof}

\begin{corollary}[Theorem \ref{thm:main}]

    $\cP$ has IDP. 
\end{corollary}
\begin{proof}
    This follows from Theorem \ref{thm:snp} and Theorem \ref{thm:IDP}.
\end{proof}

% In addition to previous results, the proof for this statement involves arguing that for $(x,y,z)\in \conv(\lambda,\mu,\iota)$, if $\alpha:=(t-(t\lambda_3-z))\lambda+(t\lambda_3-z)\mu$, we have $(x,y,z)\in \Newt(s_{\alpha})$. Seeing this requires a bit of algebraic simplification. For instance, one can verify that the third coordinate of $\alpha$ is $z$. Thus, for $(x,y,z)$ to be in $\Newt(s_{\alpha})$, it suffices to demonstrate that $x$ is less than the first coordinate of $\a$, which follows from observations we omit here. 

% \begin{proposition}
%     Case $|n_\lambda-n_\mu|=1$
% \end{proposition}
% \begin{proof}
%     Suppose $S=S_{\lambda}+S_{\mu}$, where $\lambda$ and $\mu$ are partitions with sizes $n_\lambda$ and $n_\mu$ respectively so that $|n_\lambda-n_\mu|\leq 1$. Further assume that $\lambda$ and $\mu$ are 

%     Case 1: Say $n_\lambda=n_\mu$. 
% \end{proof}

% \begin{proposition}
%     Case $|n_\lambda-n_\mu|>1$, we show these partitions are not maximal.
% \end{proposition}

% \usepackage{appendix}

\appendix

\section{Sage}\label{app:sage}

The following was used to aid in computations. To be used for Lemma \ref{lem:same_size_differ_1}, we let $\lambda_1=a$, $\lambda_2=b$, $\lambda_3=c$, $\mu_1=d$, $\mu_2=e$, and $\mu_3=f$. Thus, the matrix $A$ is $[\lambda \ \mu \ \iota]$. Note in place of $n-\mu_1-\lambda_3$ we use $\lambda_1+\lambda_2-\mu_1$, keeping in mind that $n=\lambda_1+\lambda_2+\lambda_3$, and this identity was used extensively in translating and simplifying the following output.

\begin{python}
var("a","b","c","d","e","f")
A=Matrix([[a,d,d],[b,e,a+b-d],[c,f,c]])    
print(A.determinant().expand())

\end{python}
\begin{verbatim}
a*c*d - c*d^2 + a*c*e - c*d*e - a^2*f - a*b*f + a*d*f + b*d*f
\end{verbatim}
\begin{python}
B=A.inverse()
gamma=vector([d+1,a+b-d,c-1])
X=vector(B*gamma)
print(X.simplify_full())
\end{python}

\begin{verbatim}
(-((a + b)*d - d^2 - (c + d)*e + (a + b - d)*f)/
(a*c*d - c*d^2 + (a*c - c*d)*e - (a^2 + a*b - (a + b)*d)*f), 
(a + b + c)/(c*d + c*e - (a + b)*f), 
-(c*d^2 - (a*c + b)*d - ((a - 1)*c - c*d - a)*e + (a^2 + (a - 1)*b - (a + b)*d)*f)/
(a*c*d - c*d^2 + (a*c - c*d)*e - (a^2 + a*b - (a + b)*d)*f))
\end{verbatim}

\bibliographystyle{alphaurl}
\bibliography{sample}

\end{document}